\documentclass[12pt]{amsart}

\usepackage[T1]{fontenc}
\usepackage{genyoungtabtikz}
\usepackage{mathpazo,graphicx}
\usepackage[none]{hyphenat}
\usepackage{mathpple}
\usepackage{amsmath,amsfonts, pxfonts}
\usepackage{amsrefs}
\usepackage{amsthm}
\usepackage[all]{xy}
\usepackage{color}
\usepackage{tikz, tikz-cd}

\usepackage{hyperref}

\newcommand{\on}{\operatorname}

\newcommand{\mf}{\mathfrak}
\newcommand{\mc}{\mathcal}

\newcommand{\fun}{\mathbb{F}_1}

\renewcommand{\H}{{\mathcal H}}

\newcommand{\A}{\on{A}}

\newcommand{\mspec}{\on{MSpec}}

\newcommand{\p}{\mathfrak{p}}

\newcommand{\ol}{\overline}
\newcommand{\nc}{\newcommand}

\nc{\Ann}{\on{Ann}}
\nc{\Supp}{\on{Supp}}

\nc{\SK}{\mathbf{SK}}
\nc{\sk}{\mathfrak{sk}}
\nc{\skn}{\mathfrak{sk}_n}
\nc{\h}{\mathfrak{h}}
\nc{\g}{\mathfrak{g}}
\nc{\n}{\mathfrak{n}}
\nc{\m}{\mathfrak{m}}
\nc{\ch}{\on{CH}}
\nc{\wt}{\widetilde}

\nc{\F}{\mc{F}}
\nc{\C}{\mc{C}}
\nc{\M}{\on{M}}
\nc{\T}{\mc{T}}
\renewcommand{\H}{\on{H}}
\nc{\G}{\mc{G}}
\nc{\ov}{\overline}

\nc{\VFun}{\on{Vect}(\fun)}

\nc{\FF}{\mathbb{F}}

\nc{\Pone}{\mathbb{P}^1}
\nc{\Aone}{\mathbb{A}^1}
\renewcommand{\fun}{\mathbb{F}_1}
\renewcommand{\mf}{\mathfrak}
\nc{\slthat}{\widehat{\mf{sl}}_2}
\renewcommand{\a}{\mathfrak{a}}
\renewcommand{\p}{\mathfrak{p}}
\nc{\spec}{\on{MSpec}}
\nc{\Msch}{\mc{M}sch}
\nc{\mt}{\widetilde{M}}

\nc{\Pn}{\fun \langle x_1, \cdots, x_n \rangle}
\nc{\Amod}{\on{A-mod}}
\nc{\Anmod}{\Amod_{n}}
\nc{\AnNmod}{\Amod_{n}^{\alpha}}
\nc{\Zplus}{\mathbb{Z}_{\geq 0}}
\nc{\funmod}{\on{\fun-mod}}
\nc{\ra}{\rightarrow}

\nc{\supp}{\on{Supp}}
\nc{\Zn}{\mathbb{Z}^n}
\nc{\Ztwo}{\mathbb{Z}^2}

\theoremstyle{plain}
\newtheorem{theorem}{Theorem}[subsection]
\newtheorem{thm-defn}{Theorem/Definition}
\newtheorem{proposition}{Proposition}[subsection]
\newtheorem{lemma}[theorem]{Lemma}
\newtheorem{lem-defn}{Lemma/Definition}[section]

\theoremstyle{definition}
\newtheorem{rmk}[theorem]{Remark}

\newtheorem{example}{Example}
\newtheorem{definition}[theorem]{Definition}
\newtheorem*{theorem*}{Theorem}
\newtheorem*{corollary*}{Corollary}

\theoremstyle{rmk}

\allowdisplaybreaks[4]  

\begin{document}
\title{The Hopf algebra of skew shapes, torsion sheaves on $\mathbb{A}_{/ \fun}^n$, and ideals in Hall algebras of monoid representations}
  \author{ Matt Szczesny}
  \date{}
  \maketitle

\begin{abstract}
We study ideals in Hall algebras of monoid representations on pointed sets corresponding to certain conditions on the representations. These conditions include the property that the monoid act via partial permutations, that the representation possess a compatible grading, and conditions on the support of the module. Quotients by these ideals lead to combinatorial Hopf algebras which can be interpreted as Hall algebras of certain sub-categories of modules. In the case of the free commutative monoid on $n$ generators, we obtain a co-commutative Hopf algebra structure on $n$-dimensional skew shapes, whose underlying associative product amounts to a "stacking" operation on the skew shapes. The primitive elements of this Hopf algebra correspond to connected skew shapes, and form a graded Lie algebra by anti-symmetrizing the associative product. We interpret this Hopf algebra as the Hall algebra of a certain category of coherent torsion sheaves on $\mathbb{A}_{/ \fun}^n$ supported at the origin, where $\fun$ denotes the field of one element. This Hopf algebra may be viewed as an $n$-dimensional generalization of the Hopf algebra of symmetric functions, which corresponds to the case $n=1$. 
\end{abstract}

\section{Introduction}

This paper introduces a Lie algebra structure $\skn$ on $n$-dimensional connected skew shapes. The enveloping algebra $\mathbb{U}(\skn)$ is constructed as the Hall algebra of a sub-category of torsion sheaves on $\mathbb{A}^n_{/\fun}$ supported at the origin, where $\fun$ denotes the "field of one element". The connected skew shapes correspond to the indecomposable objects of this category. In the introduction below we explain how the Hall algebra construction, applied in non-additive contexts such as that of algebraic geometry over $\fun$ produces combinatorial Hopf algebras of a representation theoretic nature which can be viewed as degenerations of quantum-group like objects over $\mathbb{F}_q$. 

\subsection{Hall algebras of Abelian categories}

The study of Hall algebras is by now a well-established area with several applications in representation theory and algebraic geometry (see \cite{S} for a very nice overview). We briefly recall the generic features of the most basic version of this construction. Given an abelian category $\mc{C}$, let 
\[
{Fl}_{i}(\mc{C}) := \{ A_0 \subset A_1 \subset \cdots \subset A_i \vert \; A_k \in \on{Ob}(\C) \}
\]
denote the stack parametrizing flags of objects in $\mc{C}$ of length $i+1$ (viewed here simply as a set). Thus $${Fl}_0 (\C) = \on{Iso}(\C), $$ the moduli stack of isomorphism classes of objects of $\C$, and $${Fl}_1(\C) = \{ A_0 \subset A_1 \vert \; A_0, A_1 \in \on{Ob}(\C) \}$$ is the usual Hecke correspondence. We have maps 
\begin{equation} \label{Hecke_correspondence}
\pi_i : Fl_{1} (\C) \rightarrow Fl_{0} (\C), \; i=1, 2, 3
\end{equation}
where
\begin{align*}
\pi_1 (A_0 \subset A_1) &= A_0 \\
\pi_2 (A_0 \subset A_1) &= A_1 \\
\pi_3 (A_0 \subset A_1) &= A_1/A_0 \\
\end{align*}
We may then attempt to define the Hall algebra of $\C$ as the space of $\mathbb{Q}$-valued functions on $Fl_0 (\C)$ with finite support, i.e.
\[
\on{H}_{\C} = \mathbb{Q}_c[Fl_0 (\C)]
\]
with the convolution product defined for $f, g \in \on{H}_{\C}$
\[
f \star g := \pi_{2 *}(\pi^*_3(f) \pi^*_1 (g)),
\]
where $\pi^*_i$ denotes the usual pullback of functions and $\pi_{i *}$ denotes integration along the fiber. To make this work, one has to impose certain finiteness conditions on $\C$. The simplest, and most restrictive such condition is that $\C$ is \emph{finitary}, which means that $\on{Hom}(M,M')$ and $\on{Ext}^1(M,M')$ are finite sets for any pair of objects $M,M' \in \C$. Important examples of categories $\C$ having these properties are $\C=\on{Rep}(Q,\mathbb{F}_{q})$ - the category of representations of a quiver $Q$ over a finite field $\mathbb{F}_q$, and $\C = Coh(X/\mathbb{F}_q)$ - the category of coherent sheaves on a smooth projective variety $X$ over $\mathbb{F}_q$. In these examples, the structure constants of $\on{H}_{\C}$ depend on the parameter $q$, and $\on{H}_{\C}$ recovers (parts of) quantum groups and their higher-loop generalizations. The basic recipe for constructing $\on{H}_{\C}$ sketched here has a number of far-reaching generalizations that extend well beyond the case of $\C$ finitary (see \cite{S}). 

The case when $\C = Coh(X_{/\mathbb{F}_q})$ is especially rich, and reasonably explicit results about the structure of $\H_{\C}$ are only available when $dim(X)=1$. In this special case, the stacks $Fl_i (\C)$ are naturally the domains of definition of automorphic forms for the function field $\mathbb{F}_q (X)$, and so "live in" the Hall algebra $\H_{\C}$. Here, the theory makes contact with the Langlands program over function fields (for more on this, see the beautiful papers \cite{K1,KSV, BS, SE1}). When $dim(X) \geq 2$, the detailed structure of $\H_{\C}$ remains quite mysterious, and the technical problems are considerable. 

\subsection{Hall algebras in a non-additive setting}

A closer examination of the basic construction of $\H_{\C}$ outlined above shows that the assumption that $\C$ be Abelian is unnecessary. All that is needed to make sense of the Hecke correspondence (\ref{Hecke_correspondence}) used to define $\H_{\C}$ is a category with a reasonably well-behaved notions of kernels/co-kernels and exact sequences. For instance, the paper \cite{DK} describes a class of "proto-exact" categories, which need not be additive, and for which the basic setup above can be made to work. Working in a non-additive context is sometimes given the catch-all slogan of "working over $\fun$", where $\fun$ stands for the mythical "field of one element".  The terminology and notation stem from the observation that functions enumerating various "linear" objects over $\mathbb{F}_q$ often have a limit $q \rightarrow 1$ which enumerates some analogous combinatorial objects. This is best illustrated with the example of the Grassmannian, where the number of $\mathbb{F}_q$-points is given by a $q$-binomial coefficient:
\[
\# \on{Gr}(k,n)/\mathbb{F}_q = \frac{[n]_q !}{[n-k]_q ! [k]_q !} 
\]
where
\[
[n]_q ! = [n]_q [n-1]_q \ldots [2]_q
\]
and 
\[
[n]_q = 1 + q + q^2 + \ldots + q^{n-1}
\]
In the limit $q \rightarrow 1$ this reduces to the binomial coefficient $\binom{n}{k}$, counting $k$-element subsets of an $n$-element set. This suggests that a vector space over $\fun$ should be a (pointed) set, with cardinality replacing dimension. 

Many examples of non-additive categories $\C$ where $\H_{\C}$ can be constructed thus come from combinatorics. Here, $\on{Ob}(\C)$ typically consist of combinatorial structures equipped an operation of "collapsing" a sub-structure, which corresponds to forming a quotient in $\C$. Examples of such $\C$ include trees, graphs, posets, matroids, semigroup representations on pointed sets, quiver representations in pointed sets etc. (see \cite{KS, Sz1, Sz2, Sz3, Sz4} ). The product in $\H_{\C}$, which counts all extensions between two objects, thus amounts to enumerating all combinatorial structures that can be assembled from the two. In this case, $\H_{\C}$ is (dual to) a combinatorial Hopf algebra in the sense of \cite{LR}. Many combinatorial Hopf algebras arise via this mechanism. 

\subsection{Hall algebras over $\fun$ as degenerations of Hall algebras over $\mathbb{F}_q$}

One may ask how the $q \rightarrow 1$ limit behaves at the level of Hall algebras. Let us illustrate this via an example. Given a Dynkin quiver $Q$, we may consider both the Abelian category $\on{Rep}(Q,\mathbb{F}_q)$ and the non-additive category $\on{Rep}(Q, \mathbb{F}_1)$. The latter is defined by replacing ordinary vector spaces at each vertex by pointed sets (i.e. "vector spaces over $\fun$" - see \cite{Sz4} for details). The Lie algebra $\g_{Q}$ with Dynkin diagram (the underlying unoriented graph) $Q$ has a triangular decomposition $$ \g_{Q} = \n_{Q,-} \oplus \h \oplus \n_{Q,+} .$$ By a celebrated result of Ringel and Green, one obtains that $$ \H_{\on{Rep}(Q,\mathbb{F}_q)} \simeq \mathbb{U}_{\sqrt{q}} (\n_{Q,+}), $$ where the right hand side denotes the quantized enveloping algebra of $\n_{Q,+}$. On the other hand, the author has shown that $$ \H_{\on{Rep}(Q,\mathbb{F}_1)} \simeq \mathbb{U} (\n_{Q,+})/ \mathcal{I} $$ where on the right hand side we now have a quotient of the \emph{ordinary} enveloping algebra by a certain ideal $\mathcal{I}$, which in type $A$ is trivial. This suggests the following (unfortunately, at this stage, very vague) principle: given a finitary abelian category $\C_{q}$ linear over $\mathbb{F}_q$ and $\C_1$ a candidate for its "limit over $\fun$", $ \H_{\C_1} $ is a (possibly degenerate) limit of $\H_{\C_q}$. 

This principle is borne out in another, more geometric example, when $\C_q = Coh(\mathbb{P}^1_{/\mathbb{F}_q})$. A result of Kapranov and Baumann-Kassel shows that in this case, there is an embedding
\[
\Psi: \mathbb{U}_{\sqrt{q}} (L \mathfrak{sl}^+_2) \rightarrow \H_{\C_q},
\] 
where $L \mathfrak{sl}^+_2$ denotes a non-standard Borel of the loop algebra of $\mathfrak{sl}_2$.  The category $\C_1 = Coh(\mathbb{P}^1_{/\fun})$ is constructed using a theory of schemes over $\fun$ and coherent sheaves on them developed by Deitmar and others (see \cite{D,D2, CHWW}). In this simplest version of algebraic geometry over $\fun$, schemes are glued out of prime spectra of commutative monoids rather than rings, and coherent sheaves are defined as sheaves of set-theoretic modules over the monoid structure sheaf. In \cite{Sz1}, the author shows that  
\[
\H_{\C_1} \simeq \mathbb{U}(L \mathfrak{gl}^+_2 \oplus \kappa)
\]
where $L \mathfrak{gl}^+_2$ denotes a Borel in the loop algebra of $\mathfrak{gl}_2$ and $\kappa$ is a certain abelian cyclotomic factor. Moreover, this calculation is much easier than the corresponding one over $\mathbb{F}_q$.

When $dim(X) \geq 2$ and $X$ possesses a model over $\fun$, one might then hope to use the Hall algebra of $\C_1 = Coh(X_{/\fun})$ to learn something about the more complicated Hall algebra of $Coh(X_{/\mathbb{F}_q})$. Naively, we may hope that the latter is a $q$-deformation of the former. This paper begins this program by focussing on the simplest piece of the category $\C_1$, that of torsion sheaves supported at a point. We recall that when $Y_{/\mathbb{F}_q}$ is a smooth curve and $y \in Y$ is a closed point, the category $Tor_y \subset Coh(Y)$ of coherent sheaves supported at $y$ is closed under extensions, and yields a Hopf sub-algebra of $\H_{Coh(Y)}$ isomorphic to the Hopf algebra $\Lambda$ of symmetric functions (see \cite{S}). It is shown in \cite{Sz1} that this result holds over $\mathbb{F}_1$ as well, provided one restricts to a certain sub-category of sheaves which in this paper are called "type $\alpha$", and which in \cite{Sz1} are called "normal". From this perspective, this paper deals with the $n$-dimensional generalization of the Hopf algebra of symmetric functions. 

The monoid scheme $\mathbb{A}^n_{/\fun}$ has a single closed point - the origin, and coherent sheaves supported there form a sub-category $Tor_0 \subset Coh(\mathbb{A}^n_{/\fun})$ closed under extensions. Such sheaves correspond to pointed sets $(M,0_M)$ with a nilpotent action of the free commutative monoid $\A=\Pn$ on $n$ generators $x_1, \cdots, x_n$, and form a category $\Amod_0$ for which we may define a Hall algebra $\H_{0, \A}$.  We consider a sub-category $\Amod^{\alpha, gr}_0 \subset \Amod$ of $\A$-modules supported at $0$, which are furthermore of type $\alpha$ and admit a $\mathbb{Z}^n$-grading. The type $\alpha$ condition is imposed to recover the Hopf algebra of symmetric functions when $n=1$, and the grading, which is automatic in the case $n=1$ is imposed to dispose of certain pathological examples. We show

\begin{theorem*}[\ref{indecomposable_modules_are_skew_shapes}]
The indecomposable objects of $\Amod^{\alpha, gr}_0$ correspond bijectively to $n$-dimensional connected skew shapes. 
\end{theorem*}

Here by a skew shape we mean a finite sub-poset of $\mathbb{Z}^n$ under the partial order where $(x_1, \cdots, x_n) \leq (y_1, \cdots, y_n)$  iff $x_i \leq y_i $ for $i=1, \cdots, n$, considered up to translation. Each skew shape determines a $\Pn$-module where $x_i$ acts by moving $1$ box in the positive $i$th direction. For example, when $n=2$, the skew shape
\begin{center}
\Ylinethick{1pt}
\gyoung(;,\bullet;;,:;\bullet;;)
\end{center}
determines a $\fun \langle x_1, x_2 \rangle$-module on two generators (indicated by black dots), where $x_1$ moves one box to the right, and $x_2$ one box up, until the edge of the diagram is reached, and $0$ beyond that. 

 We may now define the Hall algebra $\mathbf{SK}_n$ of $\Amod^{\alpha, gr}_0$ by counting \emph{admissible short exact sequences} in this category. This is a quotient of $\H_{0,\A}$ above. Paraphrased, our main result says:
 
\begin{theorem*}[\ref{skew_Hall_theorem}]
The Hall algebra $\mathbf{SK}_n$ of $\Amod^{\alpha, gr}_0$ is isomorphic to the enveloping algebra $\mathbb{U}(\skn)$ of a graded Lie algebra $\skn$. $\skn$ has a basis corresponding to connected $n$-dimensional skew shapes. Furthermore, $\skn$ and $\mathbf{SK}_n$ carries an action of the symmetric group $S_n$ by Lie (resp. Hopf) algebra automorphisms. 
\end{theorem*}

The associative product in $\mathbf{SK}_n$ corresponds to a "stacking" operation on skew shapes.

\subsection{Future directions} 

By a result of Deitmar, It is known that every smooth proper toric variety $X_{/\mathbb{Z}}$ has a model over  $\fun$ (\cite{D2}) I.e. there exists a monoid scheme $X'_{/\fun}$ such that
\[
X \simeq X' \otimes_{\fun} \mathbb{C}
\]
The approach taken here to compute the Hall algebra of normal point sheaves may be extended to construct the full Hall algebra of normal coherent sheaves on $X$ purely in terms of the combinatorics of the fan of $X$ and n-dimensional partitions.  For instance, a coherent sheaf on $\mathbb{A}^3$ which is type $\alpha$, graded, and supported on the union of the coordinate axes may be visualized in terms of a $3$-D asymptotic partition like:
\begin{center}
\includegraphics[scale=0.8]{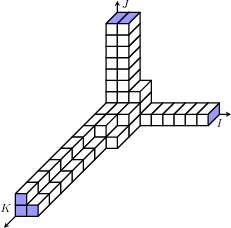}
\end{center}

This Hall algebra is once again a co-commutative Hopf algebra, and contains several copies of the Hopf algebra $\mathbf{SK}_n$ above, one for each closed point of $X'_{/\fun}$ (here $n=dim(X)$), as well as larger Hopf subalgebras corresponding to various co-dimension subvarieties of $X'{/\fun}$. This theme will be taken up in future papers \cite{Sz5, Sz6}.

\subsection{Outline}

This paper is structured as follows:

\begin{enumerate}
\item In Section \ref{MM} we recall basic properties of the category $\Amod$ of (set-theoretic) finite modules over a monoid $\A$. We discuss what is meant by exact sequences in this non-additive category, normal morphisms, as well as additional conditions which lead to sub-categories of $\Amod$ - type $\alpha$, graded, and modules with prescribed support.  
\item In Section \ref{schemes_sheaves} we briefly sketch the theory of schemes over $\fun$ via monoid schemes as introduced by Deitmar. This is done to allow a geometric interpretation of our Hall algebra in terms of coherent sheaves on $\mathbb{A}^n_{/\fun}$ supported at the origin.
\item In Section \ref{Hall_Alg} we recall the construction of the Hall algebra $\H_{\A}$ of $\Amod$, defined by counting admissible short exact sequences. We show that imposing one or several of the the conditions discussed in \ref{MM} (type $\alpha$, admitting a grading, prescribed support), yields Hopf ideals in $\H_{\A}$. $\H_{\A}$ and its various quotients by these ideals are isomorphic to enveloping algebras of Lie algebras with basis the corresponding indecomposable modules. 
\item In Section \ref{skew_torsion}, we explain the bijection between $n$-dimensional skew shapes and type $\alpha$ modules over $\Pn$ supported at $0$ and admitting a $\mathbb{Z}^n$-grading.  
\item In Section \ref{Hopf_skew} we make explicit the Lie and Hopf structures on skew shapes which emerge from the Hall algebra construction. 
\item In Section \ref{Hopf_duals} we outline structure of the Hopf duals of the Hall algebras introduced previously. 
\end{enumerate}

\noindent{\bf Acknowledgements:} The author gratefully acknowledges the support of a Simons Foundation Collaboration Grant during the writing of this paper. 

\section{Monoids and their modules} \label{MM}

A \emph{monoid} $\A$ will be an associative semigroup with identity $1_A$ and zero $0_A$ (i.e. the absorbing element). We require
\[
 1_A \cdot a = a \cdot 1_A = a \hspace{1cm} 0_A \cdot a = a \cdot 0_A = 0_A \hspace{1cm} \forall a \in A
\]
Monoid homomorphisms are required to respect the multiplication as well as the special elements $1_A, 0_A$.


\begin{example}
Let $\fun = \{ 0, 1\}$ with $$ 0 \cdot 1 = 1 \cdot 0 = 0 \cdot 0 = 0 \textrm{ and } 1 \cdot 1 = 1 .$$ We call $\fun$ \emph{the field with one element}. 
\end{example}

\begin{example}
Let $$\Pn := \{ x^{e_1}_1 x^{e_2}_2 \cdots x^{e_n}_n \vert (e_1, e_2, \cdots, e_n) \in \Zplus^n \} \cup \{ 0 \},$$ the set of monomials in $x_1, \cdots, x_n$, with usual multiplication of monomials. We will often write elements of $\Pn$ in multiindex notation as $x^e, e \in \Zplus^n$, in which case the multiplication is written as $$ x^e \cdot x^f = x^{e+f} .$$ We identify $x^0$ with $1$. 
\end{example}

$\fun$ and $\Pn$ are both commutative monoids. 

\begin{definition}
Let $\A$ be a monoid. An \emph{$\A$--module} is a pointed set $(M,0_M)$ (with $0_M \in M$ denoting the basepoint), equipped with an action of $\A$. More explicitly, an $\A$--module structure on $(M, 0_M)$ is given by a map
\begin{align*}
\A \times M  & \rightarrow M \\
(a, m) & \rightarrow a \cdot m
\end{align*}
satisfying 
\[
(a \cdot b)\cdot m = a \cdot (b \cdot m), \; \; \; 1 \cdot m = m, \; \; \; 0 \cdot m = 0_M, \; \; \; a \cdot 0_M = 0_M, \; \; \forall a,b, \in \A, \; m \in M
\]
\end{definition}
A \emph{morphism} of $\A$--modules is given by a pointed map $f: M \rightarrow N$ compatible with the action of $\A$, i.e. $f(a \cdot m) = a \cdot f(m)$. 
The $\A$--module $M$ is said to be \emph{finite} if $M$ is a finite set, in which case we define its \emph{dimension} to be $dim(M) = \vert M \vert -1$ (we do not count the basepoint, since it is the analogue of $0$). We say that $N \subset M$ is an \emph{$\A$--submodule} if it is a (necessarily pointed) subset of $M$ preserved by the action of $\A$. $\A$ always posses the trivial module $\{0\}$, which will be referred to as the \emph{zero module}.
\medskip

\noindent {\bf Note:} This structure is called an \emph{$\A$-act} in \cite{KKM} and an \emph{$\A$-set} in \cite{CLS}. 

\medskip

We denote by $\Amod$ the category of finite $\A$--modules. It is the $\fun$ analogue of the category of a finite-dimensional representations of an algebra. Note that for $M \in \Amod$, $\on{End}_{\Amod}(M) := \on{Hom}_{\Amod}(M,M)$ is a monoid (in general non-commutative). An $\fun$--module is simply a pointed set, and will be referred to as a vector space over $\fun$. Thus, an $\A$--module structure on $M \in \funmod$ amounts to a monoid homomorphism $A \rightarrow \on{End}_{\funmod}(M)$. 

\medskip

Given a morphism $f: M \rightarrow N$ in $\Amod$, we define the \emph{image} of $f$ to be $$Im(f) := \{ n \in N \vert \exists m \in M, f(m) = n \}.$$ For $M \in \Amod$ and an $\A$--submodule $N \subset M$, the \emph{quotient} of $M$ by $N$, denoted $M/N$ is the $\A$--module $$ M/N :=  M \backslash N \cup \{0 \}, $$ i.e. the pointed set obtained by identifying all elements of $N$ with the base-point, equipped with the induced $\A$--action. 

\bigskip

We recall some properties of $\Amod$,  following \cite{KKM, CLS, Sz2}, where we refer the reader for details:

\medskip

\begin{enumerate}
\item For $M,N \in \Amod$, $\vert Hom_{\Amod}(M,N) \vert < \infty$ \label{part1}
\item The trivial $\A$--module $0$ is an initial, terminal, and hence zero object of $\Amod$. 
\item Every morphism $f: M \rightarrow N$ in $C_A$ has a kernel $Ker(f):=f^{-1}(0_N)$.
\item  Every morphism $f: M \rightarrow N$ in $C_A$ has a cokernel $Coker(f):=M/Im(f)$. 
\item The co-product of a finite collection $\{ M_i \}, i \in I$ in $\Amod$ exists, and is given by the wedge product
$$
\bigvee_{i \in I} M_i = \coprod M_i / \sim
$$
where $\sim$ is the equivalence relation identifying the basepoints. We will denote the co-product of $\{M_i \}$ by $$\oplus_{i \in I} M_i$$
\item The product of a finite collection $\{ M_i \}, i \in I$ in $\Amod$ exists, and is given by the Cartesian product $\prod M_i$, equipped with the diagonal $\A$--action. It is clearly associative. it is however not compatible with the coproduct in the sense that $M \times (N \oplus L) \nsimeq M \times N \oplus M \times L$.
\item The category $\Amod$ possesses a reduced version $M \wedge N$ of the Cartesian product $M \times N$, called the smash product. $M \wedge N := M \times N / M \vee N$, where $M$ and $N$ are identified with the $\A$--submodules $\{ (m,0_N) \}$ and $\{(0_M,n)\}$ of $M \times N$ respectively. The smash product inherits the associativity from the Cartesian product, and is compatible with the co-product - i.e. $$M \wedge (N \oplus L) \simeq M \wedge N \oplus M \wedge L.$$
\item $\Amod$ possesses small limits and co-limits. 
\item If $\A$ is commutative, $\Amod$ acquires a monoidal structure called the \emph{tensor product}, denoted $M \otimes_{\A} N$, and defined by
\[
M \otimes_{\A} N := M \times N / \sim_{\otimes}
\]
where $\sim_{\otimes}$ is the equivalence relation generated by $(a \cdot m, n) \sim_{\otimes} (m, a \cdot n)$ for all $a \in \A, m \in M, n \in N$. Note that $(0_M,n) = (0 \cdot 0_M, n) \sim_{\otimes} (0_M, 0 \cdot n) = (0_M,0_N)$, and likewise $(m,0_N) \sim_{\otimes} (0_M,0_N)$. This allows us to rewrite the tensor product as $M \otimes_{\A} N = M \wedge N/ \sim_{\otimes'}$, where $\sim_{\otimes'}$ denotes the equivalence relation induced on $M \wedge N$ by $\sim_{\otimes}$. We have 
\begin{align*}
M \otimes_{\A} N  & \simeq N \otimes_{\A} M, \\  (M \otimes_{\A} N) \otimes_{\A} L & \simeq M \otimes_{\A} (N \otimes_{\A} L), \\ \ M \otimes_{\A} (L \oplus N) & \simeq (M \otimes_{\A} L) \oplus (M \otimes_{\A} N). 
\end{align*}
\item Given $M$ in $\Amod$ and $N \subset M$, there is an inclusion-preserving correspondence between flags $N \subset L \subset M$ in $\Amod$ and $\A$--submodules of $M/N$ given by sending $L$ to $L/N$. The inverse correspondence is given by sending $K \subset M/N$ to $\pi^{-1} (K)$, where $\pi: M \rightarrow M/N$ is the canonical projection. This correspondence has the property that if $N \subset L \subset L' \subset M$, then $(L'/N)/(L/N) \simeq L'/L$.  \label{property9}
\end{enumerate}

\medskip

\noindent These properties suggest that $\Amod$ has many of the properties of an abelian category, without being additive. It is an example of a \emph{quasi-exact} and \emph{belian} category in the sense of Deitmar \cite{D3} and a \emph{proto-exact} category in the sense of Dyckerhof-Kapranov \cite{DK}. Let $\on{Iso}(\Amod)$ denote the set of isomorphism classes in $\Amod$, and by $\ol{M}$ the isomorphism class of $M \in \Amod$.

\bigskip

\begin{definition}

\begin{enumerate}
\item We say that $M \in \Amod$ is \emph{indecomposable} if it cannot be written as $M = N \oplus L$ for non-zero $N, L \in \Amod$. 
\item We say $M \in \Amod$ is \emph{irreducible} or \emph{simple} if it contains no proper sub-modules (i.e those different from $0$ and $M$). 
\end{enumerate}
\end{definition}

It is clear that every irreducible module is indecomposable. 
We have the following analogue of the Krull-Schmidt theorem (\cite{Sz2}): 

\begin{proposition}
Every $M \in \Amod$ can be uniquely decomposed (up to reordering) as a direct sum of indecomposable $\A$--modules. 
\end{proposition}



\begin{rmk} \label{subobj}

Suppose $M = \oplus^k_{i=1} M_i$ is the decomposition of an $\A$--module into indecomposables, and $N \subset M$ is a submodule. It then immediately follows that $N = \oplus (N \cap M_i)$. 

\end{rmk}




Given a ring $R$, denote by $R[\A]$ the semigroup ring of $\A$ with coefficients in $R$.  We obtain a base-change functor: 
\begin{equation}
\otimes_{\fun} R : \Amod \rightarrow R[\A]-mod
\end{equation}
to the category of $R[\A]$--modules defined by setting
\[
M \otimes_{\fun} R := \bigoplus_{m \in M, m \neq 0_M} R \cdot m 
\]
i.e. the free $R$--module on the non-zero elements of $M$, with the $R[\A]$--action induced from the $\A$--action on $M$.

\subsection{Exact Sequences}

A diagram $M_1 \overset{f}{\rightarrow} M_2 \overset{g}{\rightarrow} M_3 $ in $\Amod$ is said to be \emph{exact at $M_2$} if $Ker(g) = Im(f)$. A sequence of the form $$ 0 \rightarrow M_1 \rightarrow M_2 \rightarrow M_3 \rightarrow 0 $$ is a \emph{short exact sequence} if it is exact at $M_1, M_2$ and $M_3$. 

One key respect in which $\Amod$ differs from an abelian category is he fact that given a morphism $f: M \rightarrow N$, the induced morphism $M/Ker(f) \rightarrow Im(f)$ need not be an isomorphism. This defect also manifests itself in the fact that the base change functor $\otimes_{\fun} R : \Amod \rightarrow R[\A]-mod$ fails to be exact (i.e. the base change of a short exact sequence is in general no longer short exact). $\Amod$ does however contain a (non-full) subcategory which is well-behaved in this sense, and which we proceed to describe. 

\begin{definition}
A morphism $f: M \rightarrow N$ is \emph{normal} if every fibre of $f$ contains at most one element, except for the fibre $f^{-1}(0_N)$ of the basepoint $0_N \in N$. 
\end{definition}

It is straightforward to verify that this condition is equivalent to the requirement that the canonical morphism $M/Ker(f) \rightarrow Im(f)$ be an isomorphism, and that the composition of normal morphisms is normal. 

\begin{definition}
Let $\Anmod$ denote the subcategory of $\Amod$ with the same objects as $\Amod$, but whose morphisms are restricted to the normal morphisms of $\Amod$. A short exact sequence in $\Anmod$ is called \emph{admissible}.
\end{definition}

\begin{rmk} In contrast to $\Amod$, $\Anmod$ is typically neither (small) complete nor co-complete. However, $\otimes_{\fun} R$ is exact on $\Anmod$ for any ring $R$, in the sense that this functor takes admissible short exact sequences in $\Amod$ to short exact sequences in the Abelian category $R[\A]-mod$. Note that $\on{Iso}(\Amod) = \on{Iso}(\Anmod)$, since all isomorphisms are normal. 
\end{rmk}

The following simple result is proved in \cite{Sz2}:

\begin{proposition} \label{finitary} Let $\A$ be a monoid and $\Anmod$ as above.
\item Suppose $\A$ is finitely generated. For $M,N \in \Anmod$, there are finitely many admissible short exact sequences 
\begin{equation} \label{ses2}
0 \rightarrow M \overset{f}{\rightarrow} L \overset{g}{\rightarrow} N \rightarrow 0
\end{equation}
up to isomorphism. 
\end{proposition}



\begin{definition}
Let $\A$ be a monoid. The Grothendieck Group of $\A$ is $$ K_0 (\A) := \mathbb{Z}[\overline{M}] / J , \; \; \;  \ol{M} \in \on{Iso}(\Anmod), $$ where $J$ is the subgroup generated by $\ol{L} - \ol{M} - \ol{N}$ for all admissible short exact sequences (\ref{ses2}). 
\end{definition}

\subsection{Modules of type $\alpha$ and graded modules} \label{special1}

\subsubsection{Modules of type $\alpha$}

Given a monoid $\A$ and $M \in \Amod$, we may for each $a \in \A$ consider the morphism 
$$ \phi_a : M \rightarrow M $$ $$\phi_a(m) = a \cdot m$$ in $\funmod$ (note that $\phi_a$ is not a morphism in $\Amod$ unless $\A$ is commutative). 

\begin{definition}
We say that $M$ is of \emph{type} $\alpha$ if $\forall a \in \A$, the morphism $\phi_a: M \rightarrow M$ is a normal morphism in $\funmod$. 
\end{definition}

Equivalently, if $M$ is type $\alpha$, then 
\[
a \cdot m_1 = a \cdot m_2 \iff m_1 = m_2 \textrm{ OR } a \cdot m_1 = a \cdot m_2 =0 
\]
for $a \in \A$, $m_1, m_2 \in M$. 

\begin{example}
Let $\A = \fun \langle t \rangle$. An $\A$-module is simply a vector space $M$ over $\fun$ with an endomorphism $t \in \on{End}_{\funmod}(M)$. M is thus type $\alpha$ as an $\A$--module iff $t$ is a normal morphism, i.e. if $t \vert_{M \backslash t^{-1}(0_m)}$ is injective. Thus, if $M=\{0,a,b,c \} $, and the action of $t$ is defined by
\[
t(a) =0, \; \; t(b) = a, \; \; t(c) = b,
\]
then $M$ is type $\alpha$. If however, we define the action of $t$ by
\[
t(a)=b=t(b), \; \; t(c)=0,
\]
then $M$ is not of type $\alpha$. 
\end{example}

\begin{rmk}
Let $M$ be $\A$-module, and $a \in \A$, and $\phi_a$ as above. Let $M \otimes_{\fun} \mathbb{Z}$ be the $\mathbb{Z}[A]$-module obtained via base change. The action of $\phi_a$ on $M$ extends to an endomorphism $\phi^{\mathbb{Z}}_a$ of $M \otimes_{\fun} \mathbb{Z}$ which can be represented as a square $dim(M) \times dim(M)$ matrix, having all entries either $0$ or $1$, and such that each column has at most one $1$. The condition that $M$ is type $\alpha$ implies that each row also contains at most one $1$ (i.e. that $\phi_a$ is a partial permutation matrix). Note that we always have the inclusion $$Ker(\phi_a) \otimes_{\fun} \mathbb{Z} \subset Ker(\phi^{\mathbb{Z}}_a) \; \forall a \in \A.$$
$M$ being type $\alpha$ is equivalent to the requirement that it be an isomorphism $\forall a \in \A$. 
\end{rmk}

The proof of the following proposition is straightforward:

\begin{proposition} Let $\A$ be monoid. \label{normal_properties}
\begin{enumerate}
\item If $M$ is an $\A$-module of type $\alpha$ and $N \subset M$ is a submodule, then both $N$ and $M/N$ are type $\alpha$. 
\item If $M_1, M_2$ are $\A$-modules of type $\alpha$, then so are $M_1 \oplus M_2$ and $M_1 \wedge M_2$.
\item If $\A$ is commutative, and $M_1, M_2$ are $\A$-modules of type $\alpha$, then so is $M_1 \otimes_{A} M_2$. 
\end{enumerate}
\end{proposition}

\begin{rmk}
In general, an extension of type $\alpha$ modules need not be type $\alpha$. Consider the $\fun \langle t \rangle $-module $M = \{ 0, a,b,c,d \}$, with $t$ acting by
\[
t(a)=t(b)=c, \; \; t(c) = d, \; \; t(d)=0, 
\]
and the sub-module $N=\{ 0, a,c,d \}$. Then the short exact sequence
\[
0 \ra N \ra M \ra M/N \ra 0
\]
is admissible (i.e. all morphisms are normal), and $N, M/N$ are type $\alpha$, but $M$ is not. 
\end{rmk}

\begin{definition}
Let $\Anmod^{\alpha}$ denote the full subcategory of $\Anmod$ consisting of type $\alpha$ $\A$-modules and normal morphisms. 
\end{definition}

We thus have a chain of sub-categories
\[
\Anmod^{\alpha} \subset \Anmod \subset \Amod
\]

\subsubsection{Graded modules and modules admitting a grading} \label{graded_mods}

Let $\Gamma$ is a commutative monoid written additively, and $\A$ a $\Gamma$-graded monoid. By this we mean that $\A$ can be written as
\begin{equation} \label{graded_monoid}
\A = \bigoplus_{\gamma \in \Gamma} \A_{\gamma}
\end{equation}
as $\fun$ vector spaces, with $1_{\A}, 0_{\A} \in A_{0}$, and $\A_{\gamma} \cdot \A_{\delta} \subset \A_{\gamma + \delta}$. 

\begin{definition}
A \emph{graded} $\A$-module is an $\A$-module $M$ with a decomposition
\begin{equation} \label{graded_module}
M = \bigoplus_{\gamma \in \Gamma} M_{\gamma}
\end{equation}
such that $0 \in M_{0}$ and $A_{\gamma} \cdot M_{\delta} \subset M_{\gamma+\delta}$. 
\end{definition}

We note that a sub-module of a graded module is automatically homogenous. The following lemma is obvious:

\begin{lemma} Let $\A$ be a $\Gamma$-graded monoid. \label{graded_properties}
\begin{enumerate}
\item If $M$ is a graded $\A$-module and $N \subset M$ is a submodule, then both $N$ and $M/N$ inherit a canonical grading. 
\item If $M_1, M_2$ are graded $\A$-modules, then so are $M_1 \oplus M_2$ and $M_1 \wedge M_2$. 
\item If $\A$ is commutative, and $M_1, M_2$ are graded, then $M_1 \otimes_{\A} M_2$ carries a canonical grading with $deg(m_1 \otimes m_2) = deg(m_1)+deg(m_2)$. 
\end{enumerate}
\end{lemma}

Given a graded $\A$-module $M$ and $\rho \in \Gamma$, we denote by $M[\rho]$ the graded module with $M[\rho]_{\gamma} := M_{\gamma+\rho}$. We say that a module $M$ \emph{admits a grading} if it has the structure (\ref{graded_module}) for some grading. We denote by $\Anmod^{gr}$ the full sub-category of $\Anmod$ of modules admitting a grading. We do not a priori require the morphisms to be compatible with these. By \ref{graded_properties}, every morphism in $\Anmod^{gr}$ has a kernel and co-kernel. We have a chain of inclusions
\[
\Anmod^{gr} \subset \Anmod \subset \Amod. 
\]

\section{Schemes and coherent sheaves over $\fun$} \label{schemes_sheaves}

In this section, we briefly review the notions of monoid schemes and coherent sheaves on them. These form the simplest version of a theory of "algebraic geometry over $\fun$", which allows for a geometric interpretation of the main results of the paper. In the interest of brevity, we do not aim for this section to be self-contained, and refer the interested reader to \cite{D, CHWW, LS} for details.

Just as an ordinary scheme is obtained by gluing prime spectra of commutative rings, a monoid scheme is obtained by gluing prime spectra of commutative monoids.
Given a commutative monoid $\A$, an \emph{ideal} of $\A$ is a subset $\a \subset \A$ such that $\a \cdot \A \subset \a$. A proper ideal $\p \subset \A$ is \emph{prime} if $xy \in \p$ implies either $x \in \p$ or $y \in \p$. One defines the topological space $\mspec \A$  to be the set $$\mspec A := \{ \p | \p \subset \A \textrm{ is a prime ideal } \}, $$ with the closed sets of the form $$ V(\a) := \{ \p | \a \subset \p, \p \textrm{ prime } \},  $$ for ideals $\a \in \A$, together with the empty set. 

Given a multiplicatively closed subset $S \subset \A$, the \emph{localization of $\A$ by $S$}, denoted $S^{-1} \A$, is defined to be the monoid consisting of symbols $\{ \frac{a}{s} |  a \in \A, s \in S \}$, with the equivalence relation $$\frac{a}{s} = \frac{a'}{s'}  \iff \exists \; s'' \in S \textrm{ such that } as's'' = a's s'', $$ and multiplication is given by $\frac{a}{s} \times \frac{a'}{s'} = \frac{aa'}{ss'} $. 

For $f \in \A$, let $S_f$ denote the multiplicatively closed subset $\{ 1, f, f^2, f^3, \cdots,  \}$. We denote by $\A_f$ the localization $S^{-1}_f \A$, and by $D(f)$ the open set $\mspec \A \backslash V(f) \simeq \spec \A_f$, where $V(f) := \{ \p \in \spec \A | f \in \p \}$. The open sets $D(f)$ cover $\spec \A$. 
$\spec \A$ is equipped with a  \emph{structure sheaf} of monoids $\mc{O}_{\A}$, satisfying the property $\Gamma(D(f), \mc{O}_{\A}) = \A_f$.    Its stalk at $\p \in \mspec \A$ is $\A_{\p} := S^{-1}_{\p} \A$, where $S_{\p} = A \backslash \p$. Just as in the classical setting, the category of affine monoid schemes is equivalent to the opposite of the category of commutative monoids. From here, one can proceed to construct general monoid schemes by gluing the affines $(\mspec \A, \mc{O}_{\A})$. 

One develops the theory of quasicoherent sheaves on monoid schemes just as in the classical case, and obtains an equivalence between the category of $\A$-modules and the category of quasicoherent sheaves on $\mspec \A$. An $\A$-module $M$ gives rise to a quasicoherent sheaf $\wt{M}$ on $\mspec \A$ whose space of sections over a basic open affine $D(f)$ is
\[
\wt{M}(D(f)) := M_{f}
\]
(where $M_{f}$ is defined in the obvious way), and all quasicoherent sheaves arise this way. For an element $m \in M$, define $$\Ann_{\A} (m) := \{ a \in \A | a \cdot m = 0_M \}.$$  Obviously, $0_{\A} \subset \Ann_{\A} (m) \;  \forall m \in M$. Define $$ \Ann_{\A}(M) = \bigcap_{m \in M} \Ann_{\A}(m). $$ $\Ann_{\A}(M)$ is an ideal in $\A$, and when $M$ is finitely generated defines the scheme-theoretic support of $M$, i.e.:
\[
\Supp({M}) = \{ \p \in \mspec \A \vert {M}_{\p} \neq 0 \} = V(\Ann_{\A}(M)).
\]

We will need the following result, which is proved the same way as for commutative rings and their modules

\begin{lemma} \label{support_lemma}
Suppose that $\A$ is a commutative monoid, and that 
\[
0 \rightarrow M {\rightarrow} L {\rightarrow} N \rightarrow 0
\]
is an admissible exact sequence. Then $\Supp (L) = \Supp(M) \cup \Supp(N)$ 
\end{lemma}

This shows that $\A$-modules with support contained in a closed subset $Z \subset \mspec \A$ are closed under admissible extensions.

\section{The Hall algebra of $\Anmod$ and its quotients} \label{Hall_Alg}

Let $\A$ be a finitely generated monoid. We begin this section by reviewing the construction of the Hall algebra $\H_{\A}$ of the category $\Anmod$ following \cite{Sz2}. We then discuss Hopf ideals in $\H_{\A}$ corresponding to the conditions introduced in Section \ref{special1} and  prescribed support in Section \ref{schemes_sheaves}. For more on Hall algebras in the classical setting of Abelian categories, we refer the reader to \cite{S}. 


\medskip
 
As a vector space:
\begin{equation*} 
\H_{\A} := \{ f: \on{Iso}(\Anmod) \rightarrow \mathbb{Q} \; \vert \; \# \on{supp} (f) < \infty \}.
\end{equation*}
We equip $\H_{\A}$ with the convolution product
\begin{equation*} \label{Hall_prod}
f \star g (\ol{M}) = \sum_{N \subset M} f(\ol{M/N}) g(\ol{N}),
\end{equation*}
where the sum is over all $\A$ sub-modules $N$ of $M$ (in what follows, it is conceptually helpful to fix a representative of each isomorphism class). 
Note that Lemma \ref{finitary} and the finiteness of the support of $f,g$ ensures that the sum in 
(\ref{Hall_prod}) is finite, and that $f \star g$ is again finitely supported. 

$\H_{\A}$ is spanned by $\delta$-functions $\delta_{\ol{M}} \in \H_A$ supported on individual isomorphism classes, and so it is useful to make explicit the multiplication of two such elements. We have
\begin{equation} \label{deltamult}
\delta_{\ol{M}} \star \delta_{\ol{N}} = \sum_{ \ol{R} \in \on{Iso}(\Anmod) } \mathbf{P}^R_{M,N} \delta_{\ol{R}} 
\end{equation}
where 
\[
\mathbf{P}^R_{M,N}  := \# \vert \{ L \subset R, L \simeq N, R/L \simeq M \} \vert
\]
The product
\[
\mathbf{P}^R_{M,N} \vert \on{Aut}(M) \vert \vert \on{Aut}(N) \vert
\]
counts the isomorphism classes of admissible short exact sequences of the form
\begin{equation} \label{ses}
0 \rightarrow N \rightarrow R \rightarrow M \rightarrow 0,
\end{equation}
where $ \on{Aut}(M) $ is the automorphism group of $M$. 

$\H_{\A}$ is equipped with a coproduct
\begin{equation*}
\Delta: \H_{\A} \rightarrow \H_{\A} \otimes \H_{\A}
\end{equation*}
given by
\begin{equation} \label{Hall_coprod}
\Delta(f)(\ol{M},\ol{N}) := f(\ol{M \oplus N}).
\end{equation} 
which is clearly co-commutative, as well as a natural grading by $\mathbb{Z}_{\geq 0}$ corresponding to the dimension of $M \in \Anmod$, i.e. if $dim(M)=n$, then $deg(\delta_{\ol{M}}) = n$.   

\begin{rmk}
$\H_{\A}$ also carries a grading by $K_0 (\Anmod)$ with $deg(\delta_{\ol{M}}) = [M] \in K_0 (\Anmod)$
\end{rmk}

The following theorem is proved in \cite{Sz2}:

\begin{theorem}
$(\H_{\A}, \star, \Delta)$ is a graded, connected, co-commutative bialgebra. It is called the \emph{Hall algebra of $\A$--modules}. 
\end{theorem}

By the Milnor-Moore theorem, $\H_{\A}$ is a Hopf algebra isomorphic to  $\mathbb{U}(\n_{\A})$ - the universal enveloping algebra of $\n_{\A}$, where the latter is the Lie algebra of its primitive elements. The definition of the co-product implies that $\n_{\A}$ is spanned  by $\delta_{\ol{M}}$ for isomorphism classes $\ol{M}$ of indecomposable $\A$--modules, with bracket
\[
[\delta_{\ol{M}}, \delta_{\ol{N}}] = \delta_{\ol{M}} \star \delta_{\ol{N}} - \delta_{\ol{N}} \star \delta_{\ol{M}}.
\]

\subsection{Ideals and quotients of Hall algebras}

The product (\ref{deltamult}) in $\H_{\A}$ counts all admissible $\A$-module extensions of $M$ by $N$. We now proceed to define Hopf ideals in $\H_{\A}$ yielding quotients which may be interpreted as Hall algebras of certain sub-categories of $\Anmod$.

\subsubsection{$\H^{\alpha}_{\A}$}

Let
$$
\mc{J}^{\alpha}  = \on{span} \{ \delta_{\ol{M}} \vert \; M \textrm{ is not of type } \alpha \} \subset \H_{\A}. 
$$

\begin{theorem} \label{alpha_ideal}
$\mc{J}^{\alpha} \subset \H_{\A}$ is a Hopf ideal. 
\end{theorem}

\begin{proof}
$\mc{J}^{\alpha}$ is a graded subspace contained in the augmentation ideal of $\H_{\A}$. By the Milnor-Moore theorem, It suffices to show that $\mc{J}^{\alpha}$ is a bialgebra ideal, since then $\H_{\A} / \mc{J}^{\alpha}$ will again be a graded, connected, co-commutative bialgebra. 
It follows from part $(1)$ of Proposition \ref{normal_properties} that $\mc{J}^{\alpha}$ is two-sided ideal for the product $\star$.  To show that it is a co-ideal, it suffices to show that if $M$ is not of type $\alpha$, then
\begin{equation}
 \Delta(\delta_{\ol{M}}) \subset  H_{\A} \otimes \mc{J}^{\alpha} + \mc{J}^{\alpha} \otimes H_{\A} .
\end{equation}
From the definition of the co-product (\ref{Hall_coprod}),  $\Delta(\delta_{\ol{M}})$ is a linear combination of terms of the form $\delta_{\ol{M'}} \otimes \delta_{\ol{M''}}$ where $M=M' \oplus M''$. It follows from part $(2)$ of Proposition \ref{normal_properties} that $\delta_{\ol{M'}} \in \mc{J}^{\alpha}$ or $\delta_{\ol{M''}} \in \mc{J}^{\alpha}$. 
\end{proof}

\begin{definition}
The \emph{Hall algebra of  $\A$--modules of type $\alpha$ } is the Hopf algebra $$ \H^{\alpha}_{\A} := \H_{\A}/ \mc{J}^{\alpha}.$$ 
\end{definition}

It is immediate that $\H^{\alpha}_{\A} \simeq \mathbb{U}(\n^{\alpha}_{\A})$, where the Lie algebra $$ \n^{\alpha}_{\A} = \on{span} \{ \delta_{\ol{M}} \vert M \textrm{ is indecomposable of type } \alpha \}, $$ and where
\begin{equation} \label{normal_delta_mult}
\delta_{\ol{M}} \star \delta_{\ol{N}} = \sum_{ \ol{R} \in \on{Iso}(\AnNmod) } \mathbf{P}^R_{M,N} \delta_{\ol{R}} 
\end{equation}
 
 \subsubsection{$\H^{gr}_{\A}$}
 
Suppose that $\A$ is a $\Gamma$-graded monoid as in Section \ref{graded_mods}. Let
 \[
\mc{J}^{gr}  = \on{span} \{ \delta_{\ol{M}} \vert \; M \textrm{ does not admit a grading} \} \subset \H_{\A} 
 \]
 
 \begin{theorem}
 $\mc{J}^{gr} \subset \H_{\A}$ is a Hopf ideal. 
 \end{theorem}
 
 \begin{proof}
 The proof is identical to that of Theorem \ref{alpha_ideal} except Lemma \ref{graded_properties} is used rather than Proposition \ref{normal_properties}
 \end{proof}
 
 \begin{definition}
 The \emph{ Hall algebra of modules admitting a grading } is the Hopf algebra
$$ \H^{gr}_{\A} := \H_{\A}/ \mc{J}^{gr}.$$ 
\end{definition}

We have an isomorphism $\H^{gr}_{\A} \simeq \mathbb{U}(\n^{gr}_{\A}) $ where the Lie algebra  $$ \n^{gr}_{\A} = \on{span} \{ \delta_{\ol{M}} \vert M \textrm{ is indecomposable and admits a grading }  \}, $$ and where
\begin{equation} \label{normal_delta_mult}
\delta_{\ol{M}} \star \delta_{\ol{N}} = \sum_{ \ol{R} \in \on{Iso}( \Anmod^{gr} ) } \mathbf{P}^R_{M,N} \delta_{\ol{R}} 
\end{equation}

 \begin{rmk}
 If $\A$ is a graded monoid, we may combine the two previous constructions, as $\mc{J}^{\alpha} + \mc{J}^{gr}$ is also a Hopf ideal, to obtain the quotient
 \[
 \H^{\alpha, gr}_{\A} := \H_{\A} / (\mc{J}^{\alpha} + \mc{J}^{gr})
 \]
 We then have the following diagram of Hopf algebras
 \begin{equation} \label{comm_diag}
  \begin{tikzcd}
    \H_{\A} \arrow{r} \arrow[swap]{d} & \H^{\alpha}_{A} \arrow{d} \\
     \H^{gr}_{\A} \arrow{r} & \H^{\alpha, gr}_{\A} \\
  \end{tikzcd}
\end{equation}
 
 \end{rmk}

 \subsection{Hall algebras of modules with prescribed support}
 
Suppose now that $\A$ is a finitely generated commutative monoid, and let $Z \subset \mspec \A$ be a closed subset. By Lemma \ref{support_lemma}, the $\A$-modules with support contained in $Z$ form a sub-category $\Amod_{Z,n} \subset \Anmod$ closed under extensions. Likewise, we obtain a sub-category $\Amod^{\alpha}_{Z,n} \subset \AnNmod$ by restricting to modules of type $\alpha$ with support in $Z$.  Let
\begin{equation} \label{Zsupp1}
\H_{Z, \A} := \on{span} \{ \delta_{\ol{M}} \in \H_{\A} \vert \on{supp}(M) \subset Z \}
\end{equation}
\begin{equation} \label{Zsupp2}
\H^{\alpha}_{Z,\A} := \on{span} \{ \delta_{\ol{M}} \in \H^{\alpha}_{\A} \vert \on{supp}(M) \subset Z \}
\end{equation}
If $\A$ is graded, we can also form
\begin{equation} \label{Zsupp3}
\H^{gr}_{Z,\A} := \on{span} \{ \delta_{\ol{M}} \in \H^{gr}_{\A} \vert \on{supp}(M) \subset Z \}
\end{equation}
as well as $\H^{\alpha, gr}_{Z, \A}$ etc. The following is immediate: 

\begin{proposition} \label{Hall_with_support}
Let $\A$ be a finitely generated commutative monoid, and $Z \subset \mspec \A$ a closed subset. Then $\H_{Z,\A}, \H^{\alpha}_{Z,\A}, \H^{gr}_{Z,\A}$, $\H^{\alpha, gr}_{Z, \A}$ defined as in \ref{Zsupp1}, \ref{Zsupp2}, \ref{Zsupp3} form graded, connected, co-commutative Hopf subalgebras of $\H_{\A}$, $\H^{\alpha}_{\A}$, $\H^{gr}_{Z,\A}$,  $\H^{\alpha,gr}_{\A}$ respectively. 
\end{proposition}

\begin{example}
Let $\A = \fun \langle t \rangle$. An $\A$-module $M$ may be described in terms of a directed graph with an edge from $m$ to $t \cdot m$ if $t \cdot m \neq 0$. The possible directed graphs arising this way from indecomposable $\A$--modules were classified in \cite{Sz2}. These are either rooted trees (Figure 1), or a cycle with rooted trees attached (Figure 2).  We will refer to these as type $1$ and $2$ respectively. If $M$ is of type $1$, then $\Ann_{\A}(M) = (t)$ (these modules are nilpotent), so the corresponding coherent sheaf is supported at the origin, whereas if $M$ is of type $2$, $\Ann_{\A}(M)=0$, so the support is $\mathbb{A}^1_{/\fun}$. 
\begin{enumerate}
\item The indecomposable modules of type $\alpha$ are  those whose graphs have the property that every vertex has at most one incoming edge. These are the vertical ladders (type $1$), and the directed polygons (type $2$). 
\item If we give $\A$ the grading where $t$ has degree $1$, only type $1$ modules admit a grading.
\end{enumerate}

As shown in \cite{Sz2, Sz4} the corresponding Hall algebras can be described as follows

\begin{enumerate}
\item We have that $\H_{0, \A} = \H^{gr}_{\A} = \H^{gr}_{0,\A}$. This Hopf algebra is dual to the Connes-Kreimer Hopf algebra of rooted forests. 
\item $\H^*_{\A}$, the Hopf dual of the full Hall algebra, is an extension of the Connes-Kreimer algebra by cycles of type $2$.
\item $\H^{\alpha}_{0,\A} = \H^{\alpha,gr}_{\A} = \H^{\alpha,gr}_{0,\A}$ is the Hopf algebra of symmetric functions, with the $n$-vertex ladder corresponding to the $n$-th power sum. 
\end{enumerate}

\end{example}

\begin{minipage}{.5\textwidth}
\begin{center}
\begin{tikzpicture}
\draw [ultra thick,->] (0,0) -- (0.9,0.9);
\draw [fill] (0,0) circle [radius=0.1];
\draw [fill] (1,1) circle [radius=0.1];
\draw [ultra thick,->] (2,0) -- (1.1,0.9);
\draw [fill] (2,0) circle [radius=0.1];
\draw [ultra thick,->] (1,1) -- (1.9,1.9);
\draw [fill] (2,2) circle [radius=0.1];
\draw [fill] (4,0) circle [radius=0.1];
\draw [fill] (3,1) circle [radius=0.1];
\draw [ultra thick,->] (4,0) -- (3.1, 0.9);
\draw [ultra thick,->] (3,1) -- (2.1,1.9);
\draw [fill] (2,1) circle [radius=0.1];
\draw [ultra thick,->] (2,1) -- (2,1.9);
\draw [fill] (2,3) circle [radius=0.1];
\draw [ultra thick,->] (2,2) -- (2,2.9);
\draw [fill] (3,2) circle [radius=0.1];
\draw [ultra thick,->] (3,2) -- (2.1,2.9);
\draw [ultra thick,->] (2,3) -- (2,3.9);
\draw [fill] (2,4) circle [radius=0.1];
\node at (2,-1) {Figure 1};
\end{tikzpicture}
\end{center}
\end{minipage}
\begin{minipage}{.5\textwidth}
\begin{center}
\begin{tikzpicture}
\draw [ultra thick,->] (0,0) -- (0.9,0);
\draw [ultra thick,->] (1,0) -- (1,0.9);
\draw [ultra thick,->] (1,1) -- (0.1,1);
\draw [ultra thick,->] (0,1) -- (0,0.1);
\draw [ultra thick,->] (0,-1) -- (0,-0.1);
\draw [ultra thick,->] (-1,-1) -- (-0.1,-0.1);
\draw [ultra thick,->] (2,2) -- (1.1,1.1);
\draw [ultra thick,->] (2,3) -- (2,2.1);
\draw [ultra thick,->] (3,2) -- (2.1,2);
\draw [fill] (0,0) circle [radius=0.1];
\draw [fill] (1,1) circle [radius=0.1];
\draw [fill] (1,0) circle [radius=0.1];
\draw [fill] (0,1) circle [radius=0.1];
\draw [fill] (-1,-1) circle [radius=0.1];
\draw [fill] (0,-1) circle [radius=0.1];
\draw [fill] (2,2) circle [radius=0.1];
\draw [fill] (2,3) circle [radius=0.1];
\draw [fill] (3,2) circle [radius=0.1];
\node at (1,-2) {Figure 2};
\end{tikzpicture}
\end{center}
\end{minipage}

\section{Skew shapes and torsion sheaves on $\mathbb{A}^n_{/ \fun}$} \label{skew_torsion}

In this section, we will be exclusively focussed on the monoid $\Pn$. It is naturally graded by $\mathbb{Z}^n_{\geq 0}$, with $deg(x_i) = e_i$, where $e_i$ denotes the $i$th standard basis vector in $\mathbb{Z}^n$. We have $$ \mathbb{A}^n_{/ \fun} = \mspec \Pn .$$

\subsection{n-dimensional skew shapes}

We begin by introducing a natural partial order on $\mathbb{Z}^n$, where for $$ x=(x_1, \cdots, x_n) \in \mathbb{Z}^n  \textrm{ and } y = (y_1, \cdots, y_n) \in \mathbb{Z}^n, $$ $$ x \leq y \iff x_i \leq y_i \textrm{ for } i=1, \cdots, n. $$

\begin{definition}
An \emph{n-dimensional skew shape} is a finite convex sub-poset  $S \subset \mathbb{Z}^n$. $S$ is \emph{connected} iff the corresponding poset is. 
We consider two skew shapes $S, S'$ to be equivalent iff $S'$ is a translation of $S$, i.e. if there exists $a \in \mathbb{Z}^n$ such that $S'=a+S$. 
\end{definition}

The condition that $S$ is connected is easily seen to be equivalent to the condition that any two elements of $S$ can be connected via a lattice path lying in $S$. The name \emph{skew shape} is motivated by the fact that for $n=2$, a connected skew shape in the above sense corresponds (non-uniquely) to a difference $\lambda \backslash \mu$ of two Young diagrams in French notation. 

\begin{example} \label{example_a}
Let $n=2$, and $$S \subset \mathbb{Z}^2 = \{ (1,0),(2,0),(3,0),(0,1),(1,1), (0,2) \}$$ (up to translation by $a \in \mathbb{Z}^2).$ Then $S$ corresponds to the connected skew Young diagram  

\begin{center}
\Ylinethick{1pt}
\gyoung(;,;;,:;;;)
\end{center}
\end{example}

\subsection{Skew shapes as modules}

Let $S \subset \mathbb{Z}^n$ be a skew shape. We may attach to $S$ a $\Pn$-module $M_S$ with underlying set
\[
M_S = S \sqcup \{ 0 \},
\]
and action of $\Pn$ defined by
\begin{align*}
x^e \cdot s &= \begin{cases} s+ e, & \mbox{if } s+e \in S \\ 0 & \mbox{otherwise} \end{cases} \\
& e \in \mathbb{Z}^n_{\geq 0}, s \in {S}.
\end{align*}
 It is clear that $M_S$ is a $\Pn$-module of type $\alpha$, and that $\m^k \cdot M_S =0$ for $k$ sufficiently large, where $\m = (x_1, \cdots, x_n)$ is the maximal ideal.  

\begin{example} Let $S$ as in Example \ref{example_a}. $x_1$ (resp. $x_2$) act on the $\fun \langle x_1, x_2 \rangle$-module $M_S$ by moving one box to the right (resp. one box up) until reaching the edge of the diagram, and $0$ beyond that. A minimal set of generators for $M_S$ is indicated by the black dots:

\begin{center}
\Ylinethick{1pt}
\gyoung(;,\bullet;;,:;\bullet;;)
\end{center}
\end{example}

The proof of the following lemma is straightforward:

\begin{lemma} \label{shapes_modules}
Let $S \subset \mathbb{Z}^n$ be a skew shape, and $M_S$ the corresponding $\Pn$-module. Then the following hold:
\begin{enumerate}
\item $\on{Supp} M_S = 0 \subset \mathbb{A}^n$.
\item $M_S$ admits a grading, unique up to translation by $\mathbb{Z}^n$. 
\item $M_S$ is indecomposable iff $S$ is connected.
\item The decomposition of $S$ into connected components $$S=S_1 \sqcup S_2 \cdots \sqcup S_k$$ corresponds to the decomposition of $M_S$ into indecomposable factors $$M_S = M_{S_1} \oplus M_{S_2} \oplus \cdots \oplus M_{S_k}. $$
\item The minimal elements of $S$ form a unique minimal set of generators for $M_S$.
\item $M_S$ is cyclic (i.e. generated by a single element) iff $S$ corresponds to an ordinary $n$-dimensional Young diagram. 
\item If $T \subset S$ is a sub-poset, then $M_T$ is a sub-module of $M_S$ iff $S \backslash T$ is an order ideal. 
\item If $M_T \subset M_S$ is a sub-module corresponding to the sub-poset $T \subset S$, then the quotient $M_S/M_T \simeq M_{S \backslash T}$. 
\end{enumerate}
\end{lemma}

\begin{rmk}
It follows from part (6) of Lemma \ref{shapes_modules} that sub-modules of $M_S$ correspond to (not necessarily connected) skew shapes aligned along the outer edge of $S$. 
\end{rmk}

\begin{rmk}
Cyclic $\Pn$-modules can be thought of as the Hilbert Scheme of points on $\mathbb{A}^n/\fun$, and correspond to ordinary $n$-dimensional Young diagrams. This is consistent with the fact that these parametrize the torus fixed-points of $(\mathbb{C}^*)^n$ on the Hilbert scheme of points on $\mathbb{A}^n/\mathbb{C}$, which correspond to monomial ideals in $\mathbb{C}[x_1, \cdots, x_n]$. 
\end{rmk}

\begin{example} Let $S$ be the skew shape below, with $T \subset S$ corresponding to the sub-shape whose boxes contain t's.

\begin{center}
\Ylinethick{1pt}
\gyoung(;t,;;;t,:;;t;t,::;;;t)
\end{center}
\Ylinethick{1pt}
We have $$M_T = \gyoung(;) \oplus \gyoung(;) \oplus \gyoung(;,;;)$$
and
$$ M_S / M_T = \gyoung(;;,:;) \oplus \gyoung(;;) $$
\end{example}

\subsection{Tableaux and filtrations} 

When $n=2$, standard and semi-standard skew tableaux may be interpreted as filtrations on $\fun \langle x_1, x_2 \rangle$-modules. 

\begin{definition} let $S$ be a skew shape with $m$ boxes.
\begin{enumerate}
\item A \emph{standard tableau} with shape $S$ is a filling of $S$ with $1,2, \cdots,m$ such that the entries are strictly increasing left to right in each row and bottom to top in each column.
\item A \emph{semistandard tableau} with shape $S$ is a filling of $S$ with positive integers such that the entries are weakly increasing left to right in each row and strictly increasing bottom to top in each column.
\end{enumerate}
\end{definition}

For $s \in S$, we denote by  $f(s) \in \mathbb{Z}$ the integer assigned to that box, and refer to a tableau by the pair $(S,f)$. 
Let $S_{\geq k} = \{ s \in S \vert  f(s) \geq k \}$. When $S$ is standard or semi-standard, each $S_{\geq k}$ is easily seen to be a skew sub-shape corresponding to a sub-module $M_{S_{\geq k}} \subset M_S$. We have a (necessarily finite) filtration
\[
\cdots M_{S_{\geq r}} \subset M_{S_{\geq r-1}} \subset M_{S_{\geq r-2}} \subset \cdots
\] \Ylinethick{1pt}
The requirement that $(S,f)$ be standard is easily seen to be equivalent to the requirement that
\[
M_{S_r}/M_{S_{r-1}} \simeq \gyoung(;)
\]
whereas semi-standard $(S,f)$ correspond to filtrations with
\[
M_{S_{\geq r-1}}/M_{S_{\geq r}} \simeq \gyoung(;;_2\hdts;) \oplus \gyoung(;;_2\hdts;) \oplus \cdots
\]
i.e. the associated graded consists of direct sums of modules with underlying shapes horizontal strips. Geometrically, these correspond to coherent sheaves supported on the $x_1$-axis in $\mathbb{A}^2_{/\fun}$. 

\subsection{ Type $\alpha$ modules admitting a $\mathbb{Z}^n$ grading are connected skew shapes}

In this section we show that indecomposable type $\alpha$ modules over $\Pn$ admitting a $\mathbb{Z}^n$ grading and supported at $0 \subset \mathbb{A}^n_{/\fun}$ arise from connected skew shapes. We begin with a lemma:

\begin{lemma} \label{filtration}
Let $M$ be a finite $\Pn$-module supported at $0 \subset \mathbb{A}^n_{/\fun}$. Then $M$ has a finite filtration
\[
M \supset \m \cdot M \supset \m^2 \cdot M \supset \cdots \supset \m^r \cdot M = 0_M
\]
where $r \leq \dim(M)$.
\end{lemma}

\begin{proof}
The hypothesis implies that  the radical of $\Ann(M)$ is  $\m$. Since $\m$ is finitely generated, we have that $\m^s \subset \Ann(M)$ for some $s \in \mathbb{N}$, which implies that $\m^s \cdot M = 0_M$ for sufficiently large $s$. The bound on $r$ follows from the fact that the filtration is strictly decreasing.
\end{proof}

\begin{theorem} \label{indecomposable_modules_are_skew_shapes}
Let $M$ be a finite, $\Zn$-graded, indecomposable, type $\alpha$ $\Pn$-module such that $\supp M=0$. Then $M \simeq M_S$ for a connected skew shape $S \subset \mathbb{Z}^n$.
\end{theorem}

\begin{proof}
We proceed by induction on $d=\dim(M)$. The result is obvious for $d=0, 1$. Suppose now that it holds for $d < p = \dim(M)$. Write the decomposition into graded components as $$M = \bigoplus_{d \in \Zn} M_d.$$
By Lemma \ref{filtration}, we can find $v \in M$ such that $\m \cdot v = 0$. This implies that $v$ generates a one-dimensional $\Pn$-submodule $\langle v \rangle $ of $M$. Denote by 
$$ \pi: M \rightarrow M/ \langle v \rangle $$ 
the quotient map,  and let 
\begin{equation}
M/ \langle v \rangle = M_1 \oplus M_2 \oplus \cdots \oplus M_k
\end{equation}
be the decomposition of $M/ \langle v \rangle$ into non-zero indecomposable $\Pn$-modules. For each $i$, let $\wt{M}_i = \pi^{-1}(M_i)$. We have $M= \cup_i \wt{M}_i$, $\wt{M}_i \cap \wt{M}_j = \{ 0, v \}$ for $i \neq j$, and $\wt{M}_i/ \langle v \rangle = M_i$. $\pi$ yields a bijection between $M_i$ and the subset $\sigma(M_i) := \wt{M}_i \backslash \{ v \} \subset M$. 
By part (1) of Lemma \ref{graded_properties}, each $M_i$ is graded and satisfies condition $\alpha$, so by the induction hypothesis, $M_i \simeq M_{S_i}, \; i=1,\cdots ,k$ for connected skew shapes $S_i \subset \Zn$. 

We make the following two observations:
\begin{enumerate}
\item $v \in \m \cdot \sigma(M_i)$ for each $i$. If not, and say $v \notin \m \cdot \sigma(M_1)$, then $\sigma(M_1)$ is a proper direct summand of $M$, contradicting the hypothesis that $M$ is indecomposable. 
\item By property $\alpha$, there is at most one $v_i \in M_{deg(v) - e_i}$ such that $x_i \cdot v_i = v$, and since $M$ is graded, the $v_i$'s are pairwise distinct. 
\end{enumerate}

These observations combined show that $k \leq n$, and that each $\wt{M}_i$ contains a non-empty set $\{ v_{i,j_1}, \cdots, v_{i,j_r} \}$ of elements such that $x_{j_p} \cdot v_{i,j_p} = v$. 

We consider the case $n=2$. There are two possibilities.
\begin{enumerate}
\item $M/\langle v \rangle$ is indecomposable (i.e. $k=1$). Let $S=S_1 \sqcup deg(v) \subset \Ztwo$. There are three possibilities:
\begin{enumerate}
\item $deg(v) - e_1 \in S$ but $deg(v) - e_2 \notin S$. This means that $deg(v)$ is located on the right end of the bottom row of $S$. 
\item $deg(v) - e_2 \in S$ but $deg(v) - e_1\notin S $. This means that $deg(v)$ is located at the top of the leftmost column of $S$. 
\item $deg(v) - e_1 \in S$ and $deg(v) - e_2 \in S $. This means that $deg(v)$ is located at outer boundary of $S$ with boxes immediately to the left and below $deg(v)$. Since $S_1$ was assumed to be a connected skew shape, $deg(v) - e_1 - e_2 \in S$ as well. 
\end{enumerate}
In all three cases, $S$ is a connected skew shape, and it is clear that $M \simeq M_S$. 
\item $M/ \langle v \rangle \simeq M_1 \oplus M_2$ (i.e. $k=2$). Let $S=S_1 \sqcup S_2 \sqcup deg(v)$. We may assume, after switching the labels $M_1$ and $M_2$ if necessary, that there exist $v_1 \in M_{S_1}, v_2 \in M_{S_2}$ (necessarily unique) such that $x_1 \cdot v_1 = v$, $x_2 \cdot v_2 =v$. This means that $deg(v) - e_1 \in S_1, deg(v) - e_2 \in S_2$. $deg(v)$ is therefore located at the right end of the bottom row of $S_1$ and the top of the leftmost column of $S_2$. $S$ is therefore a connected skew shape (since $S_1, S_2$ are assumed connected skew), and $M \simeq M_S$. 
\end{enumerate}

The case of general $n$ can be dealt with using the same approach.

\end{proof}

\subsection{Action of $S_n$}

$S_n$ - the symmetric group on $n$ letters, acts by automorphisms on $\A = \Pn$ by $\sigma \cdot x_{i} = x_{\sigma(i)}$, $\sigma \in S_n$. It is easily seen to be the full automorphism group of $\A$, or equivalently, of $\mathbb{A}^n_{/\fun}$. This action induces an action on the categories $\Anmod$, $\AnNmod$, and $\Anmod^{gr}$ sending a module $M$ to the module $M^{\sigma}$ with the same underlying pointed set, and where $\A$ acts on $M^{\sigma}$ by 
\begin{equation} \label{sym_action}
a \cdot m := \sigma(a) \cdot m, \; m \in M^{\sigma}
\end{equation}
This action preserves the property of a module being indecomposable, and so by Theorem \ref{indecomposable_modules_are_skew_shapes}, induces an action on skew shapes, which is easily seen to coincide with the "classical" action. For $n=2$, this amounts to transposition of the skew shape.

\begin{example}
Let $n=2$, and $\sigma=(1 \; 2 ) \in S_2$ the only non-identity element. If
\Ylinethick{1pt}
\[
S = \gyoung(;,;;;)
\]
then
\[
S^{\sigma} = \gyoung(;,;,;;)
\]
\end{example}

\section{Hopf and Lie structures on skew shapes} \label{Hopf_skew}

Let $A=\Pn$, and $$Z  = 0 = V((x_1, \cdots, x_n)) \subset \mathbb{A}^n_{/\fun} =  \mspec \Pn$$ the origin. By Proposition \ref{Hall_with_support}, the Hall algebra $\H^{\alpha,gr}_{0,\A}$ has the structure of a graded connected co-commutative Hopf algebra, which we denote by $\mathbf{SK}_n$. By the Milnor-Moore theorem, $$ \SK_n \simeq \mathbb{U}(\skn)$$ where the Lie algebra $$ \skn = \on{span} \{ \delta_{\ol{M}} \vert M \textrm{ type } \alpha, \textrm{ indecomposable, admits a grading, } \supp M = 0  \}.$$ By Theorem \ref{indecomposable_modules_are_skew_shapes}, $\skn$ has a basis which can be identified with the set of connected $n$-dimensional skew shapes. The $S_n$ action on skew shapes \ref{sym_action} is easily seen to induce an action on $\SK_n$ by Hopf algebra automorphisms. We thus have

\begin{theorem} \label{skew_Hall_theorem}
The Hall algebra $\SK_n = \H^{\alpha,gr}_{0, \A}$ is isomorphic to the enveloping algebra $\mathbb{U}(\skn)$. The Lie algebra $\skn$ may be identified with $$ \skn = \{ \delta_{M_S} \vert S \textrm{ a connected } n- \textrm{ dimensional skew shape} \} $$ with Lie bracket 
\[
[\delta_{\ol{M_S}}, \delta_{\ol{M_T}}] = \delta_{\ol{M_S}} \star \delta_{\ol{M_T}} - \delta_{\ol{M_T}} \star \delta_{\ol{M_S}}
\]
The symmetric group $S_n$ acts on $\SK_n$ (resp. $\skn$) by Hopf (resp. Lie) algebra automorphisms. 
\end{theorem}

We recall from Section \ref{Hall_Alg} that the product in $\SK_n$ is
\begin{equation} \label{skew_mult}
\delta_{\ol{M_S}} \star \delta_{\ol{M_T}} = \sum_{ R \textrm{ a skew shape } } \mathbf{P}^R_{S,T} \delta_{\ol{M_R}} 
\end{equation}
where 
\[
\mathbf{P}^R_{S,T}  := \# \vert \{ M_L \subset M_R , M_L \simeq M_T, M_R / M_L \simeq M_S \} \vert
\]
Note that the skew shapes $R$ being summed over need not be connected. In fact, when $S$ and $T$ are connected, the product (\ref{skew_mult}) will contain precisely one disconnected skew shape $S \sqcup T$ corresponding to the split extension $M_S \oplus M_T$.

\begin{example} Let $n=2$. By abuse of notation,  we identify the skew shape $S$ with the delta-function $\delta_{\ol{M_S}} \in \mathbb{S}_2$. Let
\Ylinethick{1pt}
\[
S = \gyoung(;,;;) \; \;  \; \; \;T = \gyoung(;;)
\]
We have
\[
S \star T = \gyoung(;s,;s;s;t;t) \; + \; \gyoung(;t;t,:;s,:;s;s) \; + \; \gyoung(;s,;s;s) \oplus \gyoung(;t;t) 
\]
\[
T \star S = \gyoung(;s,;s;s,;t;t) \; + \; \gyoung(;s,;s;s,:;t;t) \; + \; \gyoung(;t;t;s,::;s;s) \; + \; \gyoung(;s,;s;s) \oplus \gyoung(;t;t) 
\]
and
\[
[S,T] =  \gyoung(;s,;s;s;t;t) \; + \; \gyoung(;t;t,:;s,:;s;s) \; - \; \gyoung(;s,;s;s,;t;t) \; - \; \gyoung(;s,;s;s,:;t;t) \; - \; \gyoung(;t;t;s,::;s;s)
\]
where for each skew shape we have indicated which boxes correspond to $S$ and $T$.
\end{example}

As the above example shows, for connected skew shapes $S,T$, the product $S \star T$ in the Hall algebra involves all ways of "stacking" the shape $T$ onto that of $S$ to achieve a skew shape, as well as one disconnected shape (which may be drawn in a number of ways). 

\begin{rmk}
it is easy to see that the structure constants of the Lie algebra $\skn$ in the basis of skew shapes are all $-1, 0,$ or $1$. 
\end{rmk}

\section{Hopf algebra duals of Hall algebras} \label{Hopf_duals}

In this section $\A$ will denote an arbitrary finitely generated monoid. 
The Hall algebras constructed in this paper are all isomorphic to enveloping algebras, where the combinatorially interesting information resides in the product. Dualizing, we obtain commutative (but in general not co-commutative) Hopf algebras whose co-product now records the interesting combinatorics. The underlying associative algebra of these duals is always a polynomial ring on a collection of indecomposable $\A$-modules. We begin with the dual of $\H_{\A}$. 

Let  $\H^*_{\A} = \mathbb{Q}[x_{\ol{M}}]$, where $\ol{M}$ runs over the distinct isomorphism classes of indecomposable $\A$-modules. As an associative algebra, $\H^*_{\A}$ is therefore a polynomial ring. We adopt the convention that if $M \in \Amod$ is written in terms of indecomposable modules as $M=M_1 \oplus M_2 \cdots \oplus M_k$, then
\[
x_{\ol{M}} := x_{\ol{M}_1} \cdot x_{\ol{M}_2} \cdots x_{\ol{M}_{k-1}} \cdot x_{\ol{M}_k}
\]
Let $$ \Delta: \H^*_{\A} \rightarrow \H^*_{\A} \otimes \H^*_{\A}$$ be defined on $x_{\ol{M}}$, $M$ indecomposable, by
\begin{equation} \label{dual_coprod}
\Delta(x_{\ol{M}}) = \sum_{N \subset M} x_{\ol{M/N}} \otimes x_{\ol{N}}
\end{equation}
and extended by the requirement that it be an algebra morphism (i.e. multiplicative). Let $$ \langle , \rangle:  \H^*_{\A} \otimes \H_{A} \rightarrow \mathbb{Q} $$ denote the pairing defined by
\[
\langle x_{\ol{M}}, \delta_{\ol{N}} \rangle = \begin{cases} 0 & M \nsimeq N \\ 1 & M \simeq N \end{cases}
\]
One can readily verify that $\H^{*}_{\A}$ with its commutative product and the co-product $\Delta$ is a graded connected commutative bialgebra, hence a Hopf algebra, and that $\langle, \rangle$ is a Hopf pairing. 

If $\A$ is graded, we may dualize the diagram (\ref{comm_diag}) to obtain:

 \begin{equation} \label{comm_diag2}
  \begin{tikzcd}
    \H^*_{\A}   &( \H^{\alpha}_{A})^* \arrow{l}  \\
    ( \H^{gr}_{\A})^* \arrow{u} & (\H^{\alpha, gr}_{\A})^* \arrow{u}  \arrow[swap]{l} \\
  \end{tikzcd}
\end{equation}

Each of the Hopf duals here may be given a description analogous to $\H^*_{\A}$. For instance, $(\H^{\alpha}_{\A})* = \mathbb{Q}[x_{\ol{M}}]$, where $\ol{M}$ runs over the isomorphism classes of indecomposable type $\alpha$ modules. The formula (\ref{dual_coprod}) remains the same. When $\A$ is commutative, and $Z \subset \mspec \A$ is a closed subset, we may incorporate support conditions in the obvious way.

\address{\tiny DEPARTMENT OF MATHEMATICS AND STATISTICS, BOSTON UNIVERSITY, 111 CUMMINGTON MALL, BOSTON} \\
\indent \footnotesize{\email{szczesny@math.bu.edu}}

\end{document}